\documentclass[11pt]{article}\UseRawInputEncoding
\usepackage{amssymb,amsfonts,amsmath,latexsym,epsf,tikz,url}
\usepackage{graphics}
\usepackage[usenames,dvipsnames]{pstricks}
\usepackage{pstricks-add}
\usepackage{epsfig}
\usepackage{pst-grad} 
\usepackage{pst-plot} 
\usepackage[space]{grffile} 
\usepackage{etoolbox} 
\makeatletter 
\patchcmd\Gread@eps{\@inputcheck#1 }{\@inputcheck"#1"\relax}{}{}
\makeatother

\newtheorem{theorem}{Theorem}[section]
\newtheorem{proposition}[theorem]{Proposition}
\newtheorem{conjecture}[theorem]{Conjecture}

\newtheorem{remark}[theorem]{Remark}

\newcommand{\qed}{\hfill $\square$\medskip}

\begin{document}

\title{Secure domination number of $k$-subdivision of graphs}

\author{
Nima Ghanbari
}

\date{May 2, 2023}

\maketitle

\begin{center}
Department of Informatics, University of Bergen, P.O. Box 7803, 5020 Bergen, Norway\\
\bigskip

{\tt Nima.ghanbari@uib.no }
\end{center}


\begin{abstract}
Let $G=(V,E)$ be a simple graph. A dominating set of $G$ is a subset $D\subseteq V$ such that every vertex not in $D$ is adjacent to at least one vertex in $D$.
The cardinality of a smallest dominating set of $G$, denoted by $\gamma(G)$, is the domination number of $G$. A dominating set $D$ is called a secure dominating set of $G$, if for every $u\in V-D$, there exists a vertex $v\in D$ such that $uv \in E$ and $(D-\{v\})\cup\{u\}$ is  a dominating set of $G$. The cardinality of a smallest secure dominating set of $G$, denoted by $\gamma_s(G)$, is the secure domination number of $G$.
 For any $k \in \mathbb{N}$, the $k$-subdivision of $G$ is a simple graph $G^{\frac{1}{k}}$ which is constructed by replacing each edge of $G$ with a path of length $k$. 
In this paper, we study the secure domination number of $k$-subdivision of $G$. 
\end{abstract}

\noindent{\bf Keywords:} domination number, secure dominating set, $k$-subdivision.

\medskip
\noindent{\bf AMS Subj.\ Class.:} 05C38, 05C69, 05C75, 05C76

\section{Introduction}

Let $G = (V,E)$ be a simple graph with $n$ vertices. Throughout this paper we consider only simple graphs.  A set $D\subseteq V(G)$ is a  dominating set if every vertex in $V(G)- D$ is adjacent to at least one vertex in $D$.
The  domination number $\gamma(G)$ is the minimum cardinality of a dominating set in $G$. There are various domination numbers in the literature.
For a detailed treatment of domination theory, the reader is referred to \cite{domination}.

\medskip 

Cockayne et al. introduced the concept of secure domination number \cite{Coc} in 2004. By their definition,
a dominating set $D$ is called a secure dominating set of $G$, if for every $u\in V-D$, there exists a vertex $v\in D$ such that $uv \in E$ and $(D-\{v\})\cup\{u\}$ is  a dominating set of $G$. The cardinality of a smallest secure dominating set of $G$, denoted by $\gamma_s(G)$, is the secure domination number of $G$. 

\medskip 

Secure domination number widely is studied in literature. In 2005, Mynhardt et al.  used a simple constructive characterisation of $\gamma$-excellent trees to obtain a constructive characterisation of trees with equal domination and secure domination numbers, where a graph $G$ is said to be $\gamma$-excellent if each vertex of $G$ is contained in some minimum dominating set of $G$ \cite{Myn}. Later, Cockayne found a sharp upper bound as $\frac{\Delta n +\Delta -1}{3\Delta -1}$, for trees  with $n$ vertices and maximum degree $\Delta \geq 3$ \cite{Coc1}. In 2008, Burger et al. by using vertex cover, showed in \cite{Bur} that if $G$ is a connected graph of order $n$ with minimum degree at least two
that is not a 5-cycle, then {{$\gamma _s(G)\leq \frac{n}{2}$}}. Merouane et al. in \cite{Mer}, showed that the problem of computing the secure domination number is in the NP-complete class, even when restricted to bipartite graphs and split graphs.  Araki et al. proposed a linear-time algorithm for finding the secure domination number of proper interval graphs in 2018 \cite{Ara}. Recently, Mohamed Ali et al. obtained the secure domination number of zero-divisor graphs \cite{Moh}. More results on secure domination number can be found in \cite{Ara1,Cas,Gro,Kis,Klo,LiXu}.

\medskip 

{{A \textit{ path  graph} is a graph whose vertices can be listed in the order $v_1, v_2, \ldots, v_n$ such that the edges are $v_i v_{i+1}$, where $i = 1, 2, \ldots, n-1$. A \textit{bipartite graph} is a set of graph vertices decomposed into two disjoint sets such that no two graph vertices within the same set are adjacent. A \textit{complete bipartite graph} is a bipartite graph such that every pair of graph vertices in the two sets are adjacent. If there are $n$ and $m$ graph vertices in the two sets, the complete bipartite graph is denoted $K_{n,m}$. A \textit{star} $S_k$ is the complete bipartite graph $K_{1,{k-1}}$.}}

\medskip 

The \textit{ $k$-subdivision} of $G$, denoted by $G^{\frac{1}{k}}$, is constructed by replacing each edge $v_iv_j$ of $G$ with a path of length $k$, say $P^{\{v_i,v_j\}}$. These $k$-paths are called \textit{superedges}, any new vertex is an internal vertex, and is denoted by $x^{\{v_i,v_j\}}_l$ if it belongs to the superedge {{$P^{\{v_i,v_j\}}$}},  with  distance $l$ from the vertex $v_i$, where $l \in \{1, 2, \ldots , k-1\}$.  Note that for $k = 1$, we have $G^{1/1}= G^1 = G$, and if  $G$ has $n$ vertices and $m$ edges, then the graph $G^{\frac{1}{k}}$ has $n+(k-1)m$ vertices and $km$ edges. Some results about subdivision of a graph can be found in \cite{ALikhani,ALikhani1,Babu}.

\medskip

In this paper, {{we study the secure domination number, and present some results on it where the graph is modified by  $k$-subdivision. }}

\section{Main Results}

In this section, we study the secure domination number of $k$-subdivision of  a graph. {{In the following, by $n$ consecutive vertices, we mean that we have  path of order $n$ as a subgarph of the graph}}. First we state some known results.

	\begin{proposition}\cite{Coc}\label{COC-pro}
For any graph $G$, $\gamma(G)\leq\gamma_s(G)$.
	\end{proposition}

	\begin{theorem}\cite{Coc}\label{COC}
{{ Let $P_n$ be a path graph  with $n$ vertices}}. Then $\gamma_s(P_n)=\lceil \frac{3n}{7} \rceil$.
	\end{theorem}

Now, we {{consider}} the $2$-subdivision of a graph and present an upper bound for its secure domination number.

	\begin{theorem}\label{G12}
{{Let $G$ be a graph which is not a star}}. Then,
 $$\gamma_s(G^{\frac{1}{2}})\leq \min \{ |E(G)|, |V(G)|\}.$$
	\end{theorem}

	\begin{proof}
Suppose that $G$ is a graph which is not star. {{ Since $G$ is not a star, then three consecutive vertices $u$, $v$, and $w$ are either in a cycle of order 3, or in a path of order 4. So we have $H_1$ and $H_2$, as we see in Figure \ref{P4C3}, as subgraphs of $G$}}. Now, we {{consider}}  $G^{\frac{1}{2}}$. Then {{each subgraph of  $G^{\frac{1}{2}}$ with six consecutive vertices, is either}} $H_1^{\frac{1}{2}}$ or $H_2^{\frac{1}{2}}$ (Figure \ref{P4C3}). {{For both cases}}, one can easily check that, $\{1,2,3\}$ is a secure dominating set for subgraphs. Now, by applying our choices {{ to all subgraphs with six consecutive vertices}} of the graph, we don't consider any vertices of $G$ in {{our set which its size is  $|E(G)|$. By}} our argument, this is a secure dominating set. On the other hand, if we consider $V(G)$ as our set for $G^{\frac{1}{2}}$, then it is a secure dominating set too.  Therefore {{$\gamma_s(G^{\frac{1}{2}})\leq \min \{ |E(G)|, |V(G)|\}$}}.
	\qed
	\end{proof}

	\begin{figure}
		\begin{center}
			\psscalebox{0.5 0.5}
{
\begin{pspicture}(0,-7.265)(16.54139,2.365)
\psdots[linecolor=black, dotsize=0.4](0.20138885,0.015)
\psdots[linecolor=black, dotsize=0.4](2.601389,0.015)
\psdots[linecolor=black, dotsize=0.4](5.001389,0.015)
\psdots[linecolor=black, dotsize=0.4](7.4013886,0.015)
\psdots[linecolor=black, dotsize=0.4](13.401389,1.615)
\psdots[linecolor=black, dotsize=0.4](11.001389,-0.785)
\psdots[linecolor=black, dotsize=0.4](15.801389,-0.785)
\psdots[linecolor=black, dotsize=0.4](0.20138885,-5.185)
\psdots[linecolor=black, dotsize=0.4](2.601389,-5.185)
\psdots[linecolor=black, dotsize=0.4](5.001389,-5.185)
\psdots[linecolor=black, dotsize=0.4](7.4013886,-5.185)
\psdots[linecolor=black, dotsize=0.4](13.401389,-3.585)
\psdots[linecolor=black, dotsize=0.4](11.001389,-5.985)
\psdots[linecolor=black, dotsize=0.4](15.801389,-5.985)
\psline[linecolor=black, linewidth=0.08](0.20138885,0.015)(7.4013886,0.015)(7.4013886,0.015)
\psline[linecolor=black, linewidth=0.08](0.20138885,-5.185)(7.4013886,-5.185)(7.4013886,-5.185)
\psline[linecolor=black, linewidth=0.08](13.401389,1.615)(11.001389,-0.785)(15.801389,-0.785)(13.401389,1.615)(13.401389,1.615)
\psline[linecolor=black, linewidth=0.08](13.401389,-3.585)(11.001389,-5.985)(15.801389,-5.985)(13.401389,-3.585)(13.401389,-3.585)
\psdots[linecolor=black, dotstyle=o, dotsize=0.4, fillcolor=white](1.4013889,-5.185)
\psdots[linecolor=black, dotstyle=o, dotsize=0.4, fillcolor=white](3.8013887,-5.185)
\psdots[linecolor=black, dotstyle=o, dotsize=0.4, fillcolor=white](6.201389,-5.185)
\psdots[linecolor=black, dotstyle=o, dotsize=0.4, fillcolor=white](12.201389,-4.785)
\psdots[linecolor=black, dotstyle=o, dotsize=0.4, fillcolor=white](14.601389,-4.785)
\psdots[linecolor=black, dotstyle=o, dotsize=0.4, fillcolor=white](13.401389,-5.985)
\rput[bl](3.4213889,-1.665){$H_1$}
\rput[bl](13.161388,-1.645){$H_2$}
\rput[bl](3.561389,-7.185){$H_1^{\frac{1}{2}}$}
\rput[bl](13.281389,-7.265){$H_2^{\frac{1}{2}}$}
\rput[bl](0.021388855,0.535){u}
\rput[bl](2.4613888,0.515){v}
\rput[bl](4.8213887,0.515){w}
\rput[bl](7.3213887,-4.765){x}
\rput[bl](7.3213887,0.555){x}
\rput[bl](0.041388854,-4.845){u}
\rput[bl](2.4613888,-4.785){v}
\rput[bl](4.8413887,-4.745){w}
\rput[bl](1.3013889,-4.725){1}
\rput[bl](3.6613889,-4.685){2}
\rput[bl](6.041389,-4.725){3}
\rput[bl](11.601389,-4.685){1}
\rput[bl](14.921389,-4.625){2}
\rput[bl](13.261389,-5.585){3}
\rput[bl](13.301389,2.175){u}
\rput[bl](10.221389,-0.845){v}
\rput[bl](16.22139,-0.785){w}
\rput[bl](13.241389,-3.005){u}
\rput[bl](10.381389,-6.105){v}
\rput[bl](16.241388,-6.045){w}
\psdots[linecolor=black, dotstyle=o, dotsize=0.4, fillcolor=white](0.20138885,0.015)
\psdots[linecolor=black, dotstyle=o, dotsize=0.4, fillcolor=white](2.601389,0.015)
\psdots[linecolor=black, dotstyle=o, dotsize=0.4, fillcolor=white](5.001389,0.015)
\psdots[linecolor=black, dotstyle=o, dotsize=0.4, fillcolor=white](7.4013886,0.015)
\psdots[linecolor=black, dotstyle=o, dotsize=0.4, fillcolor=white](13.401389,1.615)
\psdots[linecolor=black, dotstyle=o, dotsize=0.4, fillcolor=white](15.801389,-0.785)
\psdots[linecolor=black, dotstyle=o, dotsize=0.4, fillcolor=white](11.001389,-0.785)
\psdots[linecolor=black, dotstyle=o, dotsize=0.4, fillcolor=white](0.20138885,-5.185)
\psdots[linecolor=black, dotstyle=o, dotsize=0.4, fillcolor=white](2.601389,-5.185)
\psdots[linecolor=black, dotstyle=o, dotsize=0.4, fillcolor=white](5.001389,-5.185)
\psdots[linecolor=black, dotstyle=o, dotsize=0.4, fillcolor=white](7.4013886,-5.185)
\psdots[linecolor=black, dotstyle=o, dotsize=0.4, fillcolor=white](13.401389,-3.585)
\psdots[linecolor=black, dotstyle=o, dotsize=0.4, fillcolor=white](15.801389,-5.985)
\psdots[linecolor=black, dotstyle=o, dotsize=0.4, fillcolor=white](11.001389,-5.985)
\psdots[linecolor=black, dotsize=0.4](12.201389,-4.785)
\psdots[linecolor=black, dotsize=0.4](14.601389,-4.785)
\psdots[linecolor=black, dotsize=0.4](13.401389,-5.985)
\psdots[linecolor=black, dotsize=0.4](1.4013889,-5.185)
\psdots[linecolor=black, dotsize=0.4](3.8013887,-5.185)
\psdots[linecolor=black, dotsize=0.4](6.201389,-5.185)
\end{pspicture}
}
		\end{center}
		\caption{Subgraphs $H_1$ and $H_2$ in the proof of Theorem \ref{G12}} \label{P4C3}
	\end{figure}
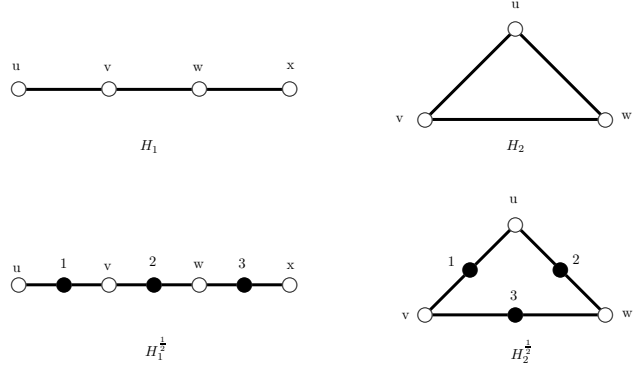

\medskip

The  condition that $G$ is not a star in Theorem \ref{G12}, is necessary.  {{Since a pendant vertex or its neighbour should be in a dominating set, then the dominating set with smallest size for $G^{\frac{1}{2}}$, where $G$ is a star, is the number of edges by choosing $x_1^{\{u,v\}}$ for all $uv\in E(G)$. Note that we have other options to have such a dominating {{set, but none of them}} are secure dominating sets. Therefore we have the following result for star graphs:}}

\begin{proposition}
For star graphs $S_n=K_{1,n-1}$, $\gamma_s(S_n^{\frac{1}{2}})=n$.
\end{proposition}

In the following, we show that both of $|V(G)|$ and $|E(G)|$ can be sharp upper bounds for Theorem \ref{G12}:

\begin{remark}
The upper bound in the Theorem \ref{G12} is {{attainable}}. First {{consider}} the path graph $P_4$. Then $\gamma_s(P_4^{\frac{1}{2}})\leq 3$ which is the number of edges of $P_4$. Also $P_4^{\frac{1}{2}}=P_7$, and by Theorem \ref{COC}, $\gamma_s(P_7)=3=|E(P_4)|$. Now, consider graph $G$, as shown in Figure \ref{graphG} {{(wheel graph)}}. {{Then}} $\{u_1,u_2,\ldots,u_7\}$ is a secure dominating set for $G^{\frac{1}{2}}$, {{and there is no set with smaller size}}, since {{in each three consecutive vertices in a path at least one vertex is in the dominating set}}. Therefore $\gamma_s(G^{\frac{1}{2}})=7=|V(G)|$.
\end{remark}

	\begin{figure}
		\begin{center}
			\psscalebox{0.5 0.5}
{
\begin{pspicture}(0,-8.225)(17.48,-0.075)
\psline[linecolor=black, linewidth=0.08](2.16,-0.725)(6.16,-0.725)(7.36,-3.525)(6.16,-6.325)(2.16,-6.325)(0.96,-3.525)(2.16,-0.725)(2.16,-0.725)
\psline[linecolor=black, linewidth=0.08](2.16,-0.725)(6.16,-6.325)(6.16,-6.325)
\psline[linecolor=black, linewidth=0.08](7.36,-3.525)(0.96,-3.525)(0.96,-3.525)
\psline[linecolor=black, linewidth=0.08](6.16,-0.725)(2.16,-6.325)(2.56,-6.325)
\psdots[linecolor=black, dotsize=0.4](2.16,-0.725)
\psdots[linecolor=black, dotsize=0.4](6.16,-0.725)
\psdots[linecolor=black, dotsize=0.4](7.36,-3.525)
\psdots[linecolor=black, dotsize=0.4](6.16,-6.325)
\psdots[linecolor=black, dotsize=0.4](4.16,-3.525)
\psdots[linecolor=black, dotsize=0.4](0.96,-3.525)
\psdots[linecolor=black, dotsize=0.4](2.16,-6.325)
\rput[bl](1.58,-0.325){$u_1$}
\rput[bl](6.28,-0.345){$u_2$}
\rput[bl](7.88,-3.665){$u_3$}
\rput[bl](6.5,-6.945){$u_4$}
\rput[bl](1.72,-6.925){$u_5$}
\rput[bl](0.0,-3.665){$u_6$}
\rput[bl](4.0,-3.005){$u_7$}
\psline[linecolor=black, linewidth=0.08](11.36,-0.725)(15.36,-0.725)(16.56,-3.525)(15.36,-6.325)(11.36,-6.325)(10.16,-3.525)(11.36,-0.725)(11.36,-0.725)
\psline[linecolor=black, linewidth=0.08](11.36,-0.725)(15.36,-6.325)(15.36,-6.325)
\psline[linecolor=black, linewidth=0.08](16.56,-3.525)(10.16,-3.525)(10.16,-3.525)
\psline[linecolor=black, linewidth=0.08](15.36,-0.725)(11.36,-6.325)(11.76,-6.325)
\psdots[linecolor=black, dotsize=0.4](11.36,-0.725)
\psdots[linecolor=black, dotsize=0.4](15.36,-0.725)
\psdots[linecolor=black, dotsize=0.4](16.56,-3.525)
\psdots[linecolor=black, dotsize=0.4](15.36,-6.325)
\psdots[linecolor=black, dotsize=0.4](13.36,-3.525)
\psdots[linecolor=black, dotsize=0.4](10.16,-3.525)
\psdots[linecolor=black, dotsize=0.4](11.36,-6.325)
\rput[bl](10.78,-0.325){$u_1$}
\rput[bl](15.48,-0.345){$u_2$}
\rput[bl](17.08,-3.665){$u_3$}
\rput[bl](15.7,-6.945){$u_4$}
\rput[bl](10.92,-6.925){$u_5$}
\rput[bl](9.2,-3.665){$u_6$}
\rput[bl](13.2,-3.005){$u_7$}
\psdots[linecolor=black, dotstyle=o, dotsize=0.4, fillcolor=white](13.36,-0.725)
\psdots[linecolor=black, dotstyle=o, dotsize=0.4, fillcolor=white](13.36,-6.325)
\psdots[linecolor=black, dotstyle=o, dotsize=0.4, fillcolor=white](11.76,-3.525)
\psdots[linecolor=black, dotstyle=o, dotsize=0.4, fillcolor=white](14.96,-3.525)
\psdots[linecolor=black, dotstyle=o, dotsize=0.4, fillcolor=white](14.38,-2.085)
\psdots[linecolor=black, dotstyle=o, dotsize=0.4, fillcolor=white](12.34,-2.085)
\psdots[linecolor=black, dotstyle=o, dotsize=0.4, fillcolor=white](15.94,-2.145)
\psdots[linecolor=black, dotstyle=o, dotsize=0.4, fillcolor=white](15.98,-4.945)
\psdots[linecolor=black, dotstyle=o, dotsize=0.4, fillcolor=white](14.32,-4.925)
\psdots[linecolor=black, dotstyle=o, dotsize=0.4, fillcolor=white](12.34,-4.945)
\psdots[linecolor=black, dotstyle=o, dotsize=0.4, fillcolor=white](10.76,-4.905)
\psdots[linecolor=black, dotstyle=o, dotsize=0.4, fillcolor=white](10.78,-2.105)
\rput[bl](13.22,-0.345){1}
\rput[bl](16.26,-2.025){2}
\rput[bl](16.26,-5.245){3}
\rput[bl](13.26,-6.925){4}
\rput[bl](10.06,-5.145){5}
\rput[bl](10.12,-2.125){6}
\rput[bl](12.6,-2.065){7}
\rput[bl](14.74,-2.165){8}
\rput[bl](15.16,-3.925){9}
\rput[bl](13.62,-5.125){10}
\rput[bl](12.54,-5.085){11}
\rput[bl](11.56,-3.205){12}
\rput[bl](13.32,-8.225){$G^{\frac{1}{2}}$}
\rput[bl](3.98,-8.165){$G$}
\end{pspicture}
}
		\end{center}
		\caption{Graph $G$ and $G^{\frac{1}{2}}$} \label{graphG}
	\end{figure}
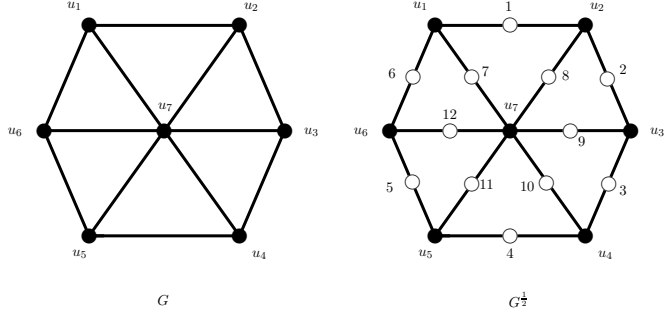

Now, we {{consider}} $G^{\frac{1}{3}}$ and present an upper and lower bound for secure domination number of {{it}}.

	\begin{theorem}\label{G13}
Let $G$ be a graph. Then,
 $$|V(G)| \leq \gamma_s(G^{\frac{1}{3}})\leq  2|E(G)|.$$
	\end{theorem}

	\begin{proof}
For every edge $uv\in E(G)$, we have a superedge $P^{\{u,v\}}$ with vertices $u$, $x_1^{\{u,v\}}$, $x_2^{\{u,v\}}$ and $v$ in $G^{\frac{1}{3}}$. To have a dominating set in four consecutive vertices, we need at least two vertices. By choosing $u$ and $v$, then we have a dominating set for $G^{\frac{1}{3}}$ with size $|V(G)|$. {{It is easy to see that there is no dominating set for $G^{\frac{1}{3}}$ with smaller size}}. Now, by Proposition \ref{COC-pro}, we have $\gamma_s(G^{\frac{1}{3}})\geq |V(G)| $. By choosing $x_1^{\{u,v\}}$ and $x_2^{\{u,v\}}$ in any superedge $P^{\{u,v\}}$,  we have a secure dominating set with size $2|E(G)|$. Therefore $\gamma_s(G^{\frac{1}{3}})\leq 2|E(G)| $ and we have the result. 
	\qed
	\end{proof}

\begin{remark}
{{The lower bound in the Theorem \ref{G13} is sharp}}. It suffices to consider star graph $S_n$. As we see in Figure \ref{graphSn3}, the set of white vertices in $S_n^{\frac{1}{3}}$ is a secure dominating set. Then $\gamma_s(S_n^{\frac{1}{3}})=n=|V(S_n)|$.
{{The upper bound in the Theorem \ref{G13} is attainable}}. It suffices to consider path graph $P_2$. Then $P_2^{\frac{1}{3}}=P_4$ and $\gamma_s(P_4)=2=2|E(P_2)|$.
\end{remark}

	\begin{figure}
		\begin{center}
			\psscalebox{0.5 0.5}
{
\begin{pspicture}(0,-5.3014426)(15.594231,0.89567304)
\psline[linecolor=black, linewidth=0.08](2.5971153,-1.7014422)(0.19711533,0.69855773)(0.19711533,0.69855773)
\psline[linecolor=black, linewidth=0.08](2.5971153,-1.7014422)(2.5971153,0.69855773)(2.5971153,0.69855773)
\psline[linecolor=black, linewidth=0.08](2.5971153,-1.7014422)(0.19711533,-1.7014422)(0.19711533,-1.7014422)
\psline[linecolor=black, linewidth=0.08](2.5971153,-1.7014422)(4.997115,0.69855773)(4.997115,0.69855773)
\psline[linecolor=black, linewidth=0.08](2.5971153,-1.7014422)(4.997115,-1.7014422)(4.997115,-1.7014422)
\psline[linecolor=black, linewidth=0.08](2.5971153,-1.7014422)(4.997115,-4.1014423)(4.997115,-4.1014423)
\psline[linecolor=black, linewidth=0.08](2.5971153,-1.7014422)(2.5971153,-4.1014423)(2.5971153,-4.1014423)
\psdots[linecolor=black, dotsize=0.1](1.7971153,-3.3014421)
\psdots[linecolor=black, dotsize=0.1](1.3971153,-2.9014423)
\psdots[linecolor=black, dotsize=0.1](0.9971153,-2.5014422)
\psdots[linecolor=black, dotsize=0.4](0.19711533,-1.7014422)
\psdots[linecolor=black, dotsize=0.4](0.19711533,0.69855773)
\psdots[linecolor=black, dotsize=0.4](2.5971153,0.69855773)
\psdots[linecolor=black, dotsize=0.4](4.997115,0.69855773)
\psdots[linecolor=black, dotsize=0.4](4.997115,-1.7014422)
\psdots[linecolor=black, dotsize=0.4](4.997115,-4.1014423)
\psdots[linecolor=black, dotsize=0.4](2.5971153,-4.1014423)
\psline[linecolor=black, linewidth=0.08](12.997115,-1.7014422)(10.5971155,0.69855773)(10.5971155,0.69855773)
\psline[linecolor=black, linewidth=0.08](12.997115,-1.7014422)(12.997115,0.69855773)(12.997115,0.69855773)
\psline[linecolor=black, linewidth=0.08](12.997115,-1.7014422)(10.5971155,-1.7014422)(10.5971155,-1.7014422)
\psline[linecolor=black, linewidth=0.08](12.997115,-1.7014422)(15.397116,0.69855773)(15.397116,0.69855773)
\psline[linecolor=black, linewidth=0.08](12.997115,-1.7014422)(15.397116,-1.7014422)(15.397116,-1.7014422)
\psline[linecolor=black, linewidth=0.08](12.997115,-1.7014422)(15.397116,-4.1014423)(15.397116,-4.1014423)
\psline[linecolor=black, linewidth=0.08](12.997115,-1.7014422)(12.997115,-4.1014423)(12.997115,-4.1014423)
\psdots[linecolor=black, dotsize=0.1](12.197115,-3.3014421)
\psdots[linecolor=black, dotsize=0.1](11.797115,-2.9014423)
\psdots[linecolor=black, dotsize=0.1](11.397116,-2.5014422)
\psdots[linecolor=black, dotsize=0.4](10.5971155,-1.7014422)
\psdots[linecolor=black, dotsize=0.4](10.5971155,0.69855773)
\psdots[linecolor=black, dotsize=0.4](12.997115,0.69855773)
\psdots[linecolor=black, dotsize=0.4](15.397116,0.69855773)
\psdots[linecolor=black, dotsize=0.4](15.397116,-1.7014422)
\psdots[linecolor=black, dotsize=0.4](15.397116,-4.1014423)
\psdots[linecolor=black, dotsize=0.4](12.997115,-4.1014423)
\psdots[linecolor=black, dotsize=0.4](12.197115,-0.9014423)
\psdots[linecolor=black, dotsize=0.4](2.5971153,-1.7014422)
\psdots[linecolor=black, dotsize=0.4](12.997115,-0.9014423)
\psdots[linecolor=black, dotsize=0.4](13.797115,-0.9014423)
\psdots[linecolor=black, dotsize=0.4](13.797115,-1.7014422)
\psdots[linecolor=black, dotsize=0.4](13.797115,-2.5014422)
\psdots[linecolor=black, dotsize=0.4](12.997115,-2.5014422)
\psdots[linecolor=black, dotsize=0.4](12.197115,-1.7014422)
\psdots[linecolor=black, dotstyle=o, dotsize=0.4, fillcolor=white](11.397116,-0.10144226)
\psdots[linecolor=black, dotstyle=o, dotsize=0.4, fillcolor=white](12.997115,-0.10144226)
\psdots[linecolor=black, dotstyle=o, dotsize=0.4, fillcolor=white](14.5971155,-0.10144226)
\psdots[linecolor=black, dotstyle=o, dotsize=0.4, fillcolor=white](14.5971155,-1.7014422)
\psdots[linecolor=black, dotstyle=o, dotsize=0.4, fillcolor=white](14.5971155,-3.3014421)
\psdots[linecolor=black, dotstyle=o, dotsize=0.4, fillcolor=white](12.997115,-3.3014421)
\psdots[linecolor=black, dotstyle=o, dotsize=0.4, fillcolor=white](11.397116,-1.7014422)
\psdots[linecolor=black, dotstyle=o, dotsize=0.4, fillcolor=white](12.997115,-1.7014422)
\rput[bl](12.5971155,-5.301442){$S_n^{\frac{1}{3}}$}
\rput[bl](2.1971154,-5.301442){$S_n$}
\end{pspicture}
}
		\end{center}
		\caption{Graph $S_n$ and $S_n^{\frac{1}{3}}$} \label{graphSn3}
	\end{figure}
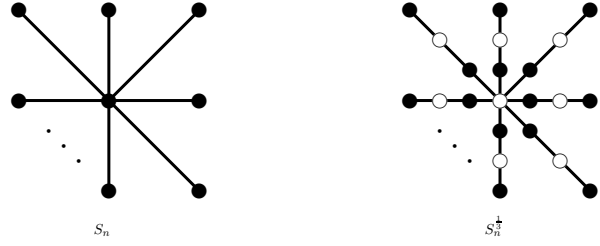

The following theorem gives the {{bounds}} of secure domination number of $4$-subdivision of  a graph. 

{{

	\begin{theorem}\label{G14}
	Let $G$ be a graph
	\begin{itemize}
	\item[(i)]
 which is not $P_2$. Then,
 $ \gamma_s(G^{\frac{1}{4}})\leq    2|E(G)|$.	
	\item[(ii)]
	with no pendant vertices. Then 
	 $ \gamma_s(G^{\frac{1}{4}})\leq    |V(G)|+|E(G)|$.
	\end{itemize}

	\end{theorem}

	\begin{proof}
		\begin{itemize}
	\item[(i)]
	 In any superedge $P^{\{u,v\}}$ in $G^{\frac{1}{4}}$, we choose $x_1^{\{u,v\}}$ and $x_3^{\{u,v\}}$ to put in our set $D$. {{It}} is a secure dominating set, because each vertex in $V-D$ can be replaced by its neighbour in $D$, if it is not $x_2^{\{u,v\}}$, and $x_2^{\{u,v\}}$ can be replaced by the neighbour which has another neighbour of degree at least 2, since $G$ is not $P_2$. Therefore we have the result.
	\item[(ii)]
	The set $D=V(G)\cup C$ is a secure dominating set, where $C$ is the central vertices of the new paths of length 4. 
	\qed	
	\end{itemize}
	\end{proof}	
}}

\medskip

Now, we {{consider}} $G^{\frac{1}{5}}$, and present upper and lower bound for secure domination number of that regarding maximum degree and number of edges of $G$.

	\begin{figure}
		\begin{center}
			\psscalebox{0.7 0.7}
{
\begin{pspicture}(0,-6.475)(17.6,2.975)
\rput[bl](5.6,2.725){$u_1$}
\rput[bl](7.2,0.325){$u_2$}
\rput[bl](0.0,-1.675){$w$}
\rput[bl](5.6,-6.475){$u_{\Delta}$}
\rput[bl](7.2,-3.675){$u_3$}
\psline[linecolor=black, linewidth=0.08](0.8,-1.675)(2.0,-1.275)(3.2,-0.875)(4.4,-0.475)(5.6,-0.075)(6.8,0.325)(6.8,0.325)
\psline[linecolor=black, linewidth=0.08](0.8,-1.675)(5.2,2.725)(5.2,2.725)
\psline[linecolor=black, linewidth=0.08](0.8,-1.675)(5.2,-6.075)(5.2,-6.075)
\psdots[linecolor=black, dotsize=0.1](6.58,-4.635)
\psdots[linecolor=black, dotsize=0.1](6.44,-4.895)
\psdots[linecolor=black, dotsize=0.1](6.22,-5.135)
\psdots[linecolor=black, dotstyle=o, dotsize=0.4, fillcolor=white](5.2,2.725)
\psdots[linecolor=black, dotstyle=o, dotsize=0.4, fillcolor=white](6.8,0.325)
\psdots[linecolor=black, dotstyle=o, dotsize=0.4, fillcolor=white](5.2,-6.075)
\rput[bl](15.6,2.725){$u_1$}
\rput[bl](17.2,0.325){$u_2$}
\rput[bl](10.0,-1.675){$w$}
\rput[bl](15.6,-6.475){$u_{\Delta}$}
\rput[bl](17.2,-3.675){$u_3$}
\psline[linecolor=black, linewidth=0.08](10.8,-1.675)(12.0,-1.275)(13.2,-0.875)(14.4,-0.475)(15.6,-0.075)(16.8,0.325)(16.8,0.325)
\psline[linecolor=black, linewidth=0.08](10.8,-1.675)(15.2,2.725)(15.2,2.725)
\psline[linecolor=black, linewidth=0.08](10.8,-1.675)(15.2,-6.075)(15.2,-6.075)
\psdots[linecolor=black, dotsize=0.1](16.58,-4.635)
\psdots[linecolor=black, dotsize=0.1](16.44,-4.895)
\psdots[linecolor=black, dotsize=0.1](16.22,-5.135)
\psdots[linecolor=black, dotstyle=o, dotsize=0.4, fillcolor=white](15.2,2.725)
\psdots[linecolor=black, dotstyle=o, dotsize=0.4, fillcolor=white](16.8,0.325)
\psdots[linecolor=black, dotstyle=o, dotsize=0.4, fillcolor=white](15.2,-6.075)
\psdots[linecolor=black, dotsize=0.4](11.6,-0.875)
\psdots[linecolor=black, dotsize=0.4](12.4,-0.075)
\psdots[linecolor=black, dotsize=0.4](14.4,1.925)
\psdots[linecolor=black, dotsize=0.4](12.0,-1.275)
\psdots[linecolor=black, dotsize=0.4](13.2,-0.875)
\psdots[linecolor=black, dotsize=0.4](15.6,-0.075)
\psline[linecolor=black, linewidth=0.08](10.8,-1.675)(16.8,-3.675)(16.8,-3.675)
\psline[linecolor=black, linewidth=0.08](0.8,-1.675)(6.8,-3.675)(6.8,-3.675)
\psdots[linecolor=black, dotstyle=o, dotsize=0.4, fillcolor=white](16.8,-3.675)
\psdots[linecolor=black, dotstyle=o, dotsize=0.4, fillcolor=white](10.8,-1.675)
\psdots[linecolor=black, dotstyle=o, dotsize=0.4, fillcolor=white](6.8,-3.675)
\psdots[linecolor=black, dotstyle=o, dotsize=0.4, fillcolor=white](0.8,-1.675)
\psdots[linecolor=black, dotsize=0.4](12.0,-2.075)
\psdots[linecolor=black, dotsize=0.4](13.2,-2.475)
\psdots[linecolor=black, dotsize=0.4](15.6,-3.275)
\psdots[linecolor=black, dotsize=0.4](11.6,-2.475)
\psdots[linecolor=black, dotsize=0.4](12.4,-3.275)
\psdots[linecolor=black, dotsize=0.4](14.4,-5.275)
\psdots[linecolor=black, dotstyle=o, dotsize=0.4, fillcolor=white](13.2,0.725)
\psdots[linecolor=black, dotstyle=o, dotsize=0.4, fillcolor=white](14.4,-0.475)
\psdots[linecolor=black, dotstyle=o, dotsize=0.4, fillcolor=white](14.4,-2.875)
\psdots[linecolor=black, dotstyle=o, dotsize=0.4, fillcolor=white](13.2,-4.075)
\rput[bl](10.4,-0.875){$x_1^{\{w,u_1\}}$}
\rput[bl](11.92,-1.095){$x_1^{\{w,u_2\}}$}
\rput[bl](12.14,-1.995){$x_1^{\{w,u_3\}}$}
\rput[bl](10.4,-3.275){$x_1^{\{w,u_{\Delta}\}}$}
\rput[bl](11.3,0.025){$x_2^{\{w,u_1\}}$}
\rput[bl](13.28,-1.435){$x_2^{\{w,u_2\}}$}
\rput[bl](13.2,-2.295){$x_2^{\{w,u_3\}}$}
\rput[bl](11.32,-4.075){$x_2^{\{w,u_{\Delta}\}}$}
\rput[bl](12.22,-4.815){$x_3^{\{w,u_{\Delta}\}}$}
\rput[bl](13.32,-5.935){$x_4^{\{w,u_{\Delta}\}}$}
\rput[bl](12.26,0.945){$x_3^{\{w,u_1\}}$}
\rput[bl](13.9,-0.275){$x_3^{\{w,u_2\}}$}
\rput[bl](14.28,-2.635){$x_3^{\{w,u_3\}}$}
\rput[bl](13.32,2.005){$x_4^{\{w,u_1\}}$}
\rput[bl](15.56,-0.715){$x_4^{\{w,u_2\}}$}
\rput[bl](15.52,-3.075){$x_4^{\{w,u_3\}}$}
\end{pspicture}
}
		\end{center}
		\caption{Vertex $w$ with maximum degree $\Delta$ and corresponding superedges in $G^{\frac{1}{5}}$ } \label{Delta15}
	\end{figure}
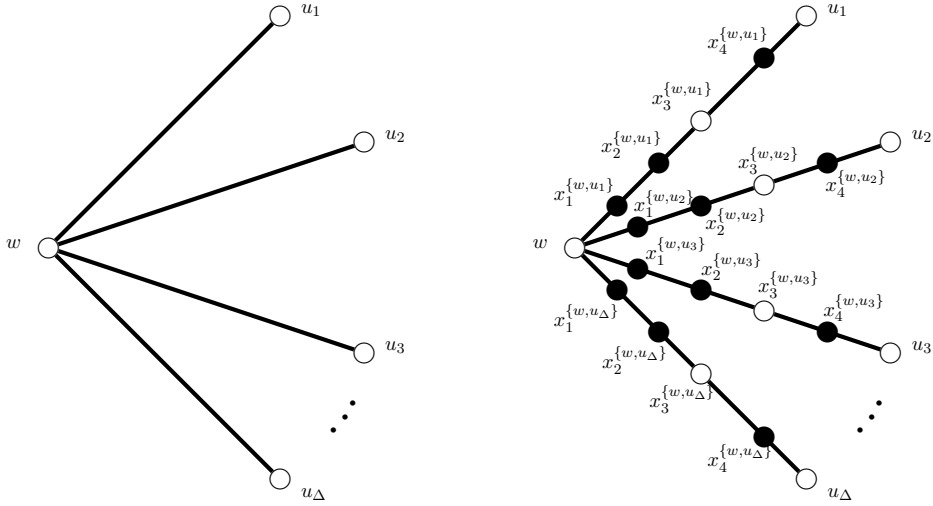

	\begin{theorem}\label{G15}
Let $G$ be a graph and $\Delta$ be the maximum degree of its vertices. Then,
 $$2|E(G)|+1 \leq \gamma_s(G^{\frac{1}{5}})\leq  3|E(G)| - \Delta +1.$$
	\end{theorem}

	\begin{proof}
For every edge $uv\in G$, we consider the superedge $P^{\{u,v\}}$ in $G^{\frac{1}{5}}$. We choose $x_1^{\{u,v\}}$,$x_2^{\{u,v\}}$ and $x_4^{\{u,v\}}$ from the set $\{u,x_1^{\{u,v\}}\,x_2^{\{u,v\}},x_3^{\{u,v\}},x_4^{\{u,v\}},v\}$ and put in a new set, {{say}} $D$. {{ So,
$$D=\bigcup_{uv \in E(G)}\Big\{x_1^{\{u,v\}},x_2^{\{u,v\}},x_4^{\{u,v\}}\Big\}.$$
}}
{{Then}} $D$ is a secure dominating set for $G^{\frac{1}{5}}$. {{Now, consider the vertex $w$}} with degree $\Delta$ (See Figure \ref{Delta15}). We define a new set $D'$ as
 $$D'=D-\{x_1^{\{w,u_1\}},x_1^{\{w,u_2\}},x_1^{\{w,u_3\}},\ldots,x_1^{\{w,u_{\Delta}\}}\} \cup \{w\}.$$
 It is easy to see that $D'$ is a secure dominating set too. Therefore
  $$\gamma_s(G^{\frac{1}{5}})\leq  3|E(G)| - \Delta +1.$$
  Now, we consider the superedges $P^{\{u,v\}}$ and $P^{\{u,t\}}$ in $G^{\frac{1}{5}}$ as we see in Figure \ref{edge15}. In every twelve consecutive vertices, we need at least four vertices to have a dominating set. We have the following cases:
\begin{itemize}
\item[(i)] We choose $x_1^{\{u,v\}}$, $x_4^{\{u,v\}}$, $x_1^{\{u,t\}}$ and $x_4^{\{u,t\}}$ to put in domination set from these superedges. {{So}} we have a dominating set for $G^{\frac{1}{5}}$ with size at most $2|E(G)|$. On the other hand, by Proposition \ref{COC-pro}, $\gamma(G^{\frac{1}{5}})\leq\gamma_s(G^{\frac{1}{5}})$ but this set is not a secure dominating set because of vertex $u$ . 
\item[(ii)]  We choose $x_2^{\{u,v\}}$, $x_4^{\{u,v\}}$, $x_1^{\{u,t\}}$ and $x_4^{\{u,t\}}$ to put in domination set from these superedges. By the same argument {{as in the previous case}}, this process does not give us a secure dominating set because of {{the}} vertex  $x_2^{\{u,t\}}$.
\item[(iii)] Put $u$ in our set. Then for having a dominating set regarding these edges, we need at least 4 other vertices and this makes our set bigger than $2|E(G)|$.
\end{itemize}    
  So $\gamma_s(G^{\frac{1}{5}}) > 2|E(G)|$. {{Hence  $\gamma_s(G^{\frac{1}{5}}) \geq 2|E(G)|+1$, and we are done.}}
	\qed
	\end{proof}

\bigskip

	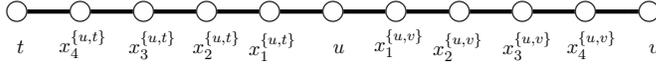
\begin{figure}
		\begin{center}
			\psscalebox{0.7 0.7}
{
\begin{pspicture}(0,-5.9693055)(12.42139,-4.827916)
\psline[linecolor=black, linewidth=0.08](0.20138885,-5.0293055)(7.4013886,-5.0293055)(7.4013886,-5.0293055)
\psline[linecolor=black, linewidth=0.08](7.4013886,-5.0293055)(11.801389,-5.0293055)(11.801389,-5.0293055)
\psdots[linecolor=black, dotstyle=o, dotsize=0.4, fillcolor=white](0.20138885,-5.0293055)
\psdots[linecolor=black, dotstyle=o, dotsize=0.4, fillcolor=white](1.4013889,-5.0293055)
\psdots[linecolor=black, dotstyle=o, dotsize=0.4, fillcolor=white](2.601389,-5.0293055)
\psdots[linecolor=black, dotstyle=o, dotsize=0.4, fillcolor=white](3.8013887,-5.0293055)
\psdots[linecolor=black, dotstyle=o, dotsize=0.4, fillcolor=white](5.001389,-5.0293055)
\psdots[linecolor=black, dotstyle=o, dotsize=0.4, fillcolor=white](6.201389,-5.0293055)
\psdots[linecolor=black, dotstyle=o, dotsize=0.4, fillcolor=white](7.4013886,-5.0293055)
\psdots[linecolor=black, dotstyle=o, dotsize=0.4, fillcolor=white](8.601389,-5.0293055)
\psdots[linecolor=black, dotstyle=o, dotsize=0.4, fillcolor=white](9.801389,-5.0293055)
\psdots[linecolor=black, dotstyle=o, dotsize=0.4, fillcolor=white](11.001389,-5.0293055)
\psline[linecolor=black, linewidth=0.08](11.801389,-5.0293055)(12.201389,-5.0293055)(12.201389,-5.0293055)
\psdots[linecolor=black, dotstyle=o, dotsize=0.4, fillcolor=white](12.201389,-5.0293055)
\rput[bl](6.201389,-5.8293056){$u$}
\rput[bl](12.201389,-5.8293056){$v$}
\rput[bl](0.20138885,-5.8293056){$t$}
\rput[bl](4.601389,-5.9693055){$x_1^{\{u,t\}}$}
\rput[bl](3.5413888,-5.9293056){$x_2^{\{u,t\}}$}
\rput[bl](2.321389,-5.9293056){$x_3^{\{u,t\}}$}
\rput[bl](1.0013889,-5.9093056){$x_4^{\{u,t\}}$}
\rput[bl](6.981389,-5.8893056){$x_1^{\{u,v\}}$}
\rput[bl](8.101389,-5.9693055){$x_2^{\{u,v\}}$}
\rput[bl](9.421389,-5.9493055){$x_3^{\{u,v\}}$}
\rput[bl](10.641389,-5.9293056){$x_4^{\{u,v\}}$}
\end{pspicture}
}
		\end{center}
		\caption{Superedges $P^{\{u,v\}}$ and $P^{\{u,t\}}$ in $G^{\frac{1}{5}}$} \label{edge15}
	\end{figure}

	\begin{figure}
		\begin{center}
			\psscalebox{0.7 0.7}
{
\begin{pspicture}(0,-6.8)(8.394231,1.594231)
\psdots[linecolor=black, dotsize=0.4](4.1971154,-2.6028845)
\psdots[linecolor=black, dotsize=0.4](4.1971154,-3.4028845)
\psdots[linecolor=black, dotsize=0.4](4.1971154,-4.2028847)
\psdots[linecolor=black, dotsize=0.4](4.1971154,-5.0028844)
\psdots[linecolor=black, dotsize=0.4](4.1971154,-5.8028846)
\psdots[linecolor=black, dotsize=0.4](4.1971154,-6.6028843)
\psdots[linecolor=black, dotsize=0.4](4.9971156,-1.8028846)
\psdots[linecolor=black, dotsize=0.4](5.7971153,-1.0028845)
\psdots[linecolor=black, dotsize=0.4](6.5971155,-0.20288453)
\psdots[linecolor=black, dotsize=0.4](7.3971157,0.59711546)
\psdots[linecolor=black, dotsize=0.4](8.197116,1.3971155)
\psdots[linecolor=black, dotsize=0.4](3.3971155,-1.8028846)
\psdots[linecolor=black, dotsize=0.4](2.5971155,-1.0028845)
\psdots[linecolor=black, dotsize=0.4](1.7971154,-0.20288453)
\psdots[linecolor=black, dotsize=0.4](0.9971155,0.59711546)
\psdots[linecolor=black, dotsize=0.4](0.19711548,1.3971155)
\psline[linecolor=black, linewidth=0.08](0.19711548,1.3971155)(4.1971154,-2.6028845)(4.1971154,-6.6028843)(4.1971154,-6.6028843)
\psline[linecolor=black, linewidth=0.08](4.1971154,-2.6028845)(8.197116,1.3971155)(8.197116,1.3971155)
\psdots[linecolor=black, dotstyle=o, dotsize=0.4, fillcolor=white](4.1971154,-2.6028845)
\psdots[linecolor=black, dotstyle=o, dotsize=0.4, fillcolor=white](7.3971157,0.59711546)
\psdots[linecolor=black, dotstyle=o, dotsize=0.4, fillcolor=white](0.9971155,0.59711546)
\psdots[linecolor=black, dotstyle=o, dotsize=0.4, fillcolor=white](4.1971154,-5.8028846)
\psdots[linecolor=black, dotstyle=o, dotsize=0.4, fillcolor=white](5.7971153,-1.0028845)
\psdots[linecolor=black, dotstyle=o, dotsize=0.4, fillcolor=white](2.5971155,-1.0028845)
\psdots[linecolor=black, dotstyle=o, dotsize=0.4, fillcolor=white](4.1971154,-4.2028847)
\end{pspicture}
}
		\end{center}
		\caption{Graph $S_4^{\frac{1}{5}}$ } \label{S415}
	\end{figure}
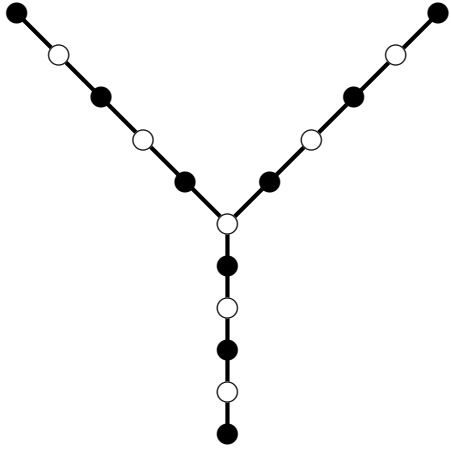

	\begin{remark}
Upper Bound in the Theorem \ref{G15} is sharp.  
{{It suffices to consider the star graph $S_n$. Then by the argument in the proof of Theorem \ref{G15}, and since from each three consecutive vertices, we need at least a vertex to be in dominating set, one can easily check that the equality holds. 
}}
\end{remark}

Now, we consider the general {{case}} $G^{\frac{1}{n}}$ for $n\geq 6$.

\begin{theorem}\label{G16}
Let $G$ be a graph and $n=7k+r$, where $k$ is a {{positive number}} and $r \in\{-1,1,3,5\}$. Then,
 $$\gamma_s(G^{\frac{1}{n}})=\gamma_s(P_{n+1})|E(G)|  .$$
	\end{theorem}

	\begin{proof}
First,	Let $n=7k+r$, where $k$ is a positive odd number and $r \in\{-1,1,3,5\}$. Consider {{the}} edge $uv\in V(G)$. As shown in Figure \ref{edge1nodd}, we consider $u$, $x_1^{\{u,v\}}$, $x_2^{\{u,v\}}$, $\ldots$, $x_{n-2}^{\{u,v\}}$, $x_{n-1}^{\{u,v\}}$, $v$ as vertex sequence of $P^{\{u,v\}}$. Now, we define
	\[
E_{uv}=\bigcup_{i=0}^{k-1} \Big\{ x_{7i+1}^{\{u,v\}},x_{7i+3}^{\{u,v\}},x_{7i+5}^{\{u,v\}}\Big\},
\]
and 
 \[
F_{uv}=\left\{
 \begin{array}{ll}
 {\displaystyle
 	\Big\{\Big\}}&
 \quad\mbox{if $r=-1$, }\\[15pt]
 {\displaystyle
 	\Big\{ x_{n-1}^{\{u,v\}} \Big\}}&
 \quad\mbox{if $r=1$,}\\[15pt]
  {\displaystyle
 	\Big\{ x_{n-3}^{\{u,v\}}, x_{n-1}^{\{u,v\}} \Big\}}&
 \quad\mbox{if $r=3$,}\\[15pt]
  {\displaystyle
 	\Big\{ x_{n-5}^{\{u,v\}}, x_{n-3}^{\{u,v\}}, x_{n-1}^{\{u,v\}} \Big\}}&
 \quad\mbox{if $r=5$.}
 \end{array}
 \right.	
 \]
 
 Let $D_{uv}=E_{uv}\cup F_{uv}$ and 
 	\[
D=\bigcup_{uv\in E(G)} D_{uv}.
\]
It is easy to see that $D$ is a secure dominating set for $G^{\frac{1}{n}}$. So $\gamma_s(G^{\frac{1}{n}})\leq |E(G)|\lceil \frac{3n+3}{7} \rceil$. Hence $\gamma_s(G^{\frac{1}{n}})\leq \gamma_s(P_{n+1})|E(G)|$ {{because of Theorem \ref{COC}}}. Now, we show that we can not use less vertices to make a secure dominating set. In the way we choose $D$, we {{cannot}} choose less vertices among $x_2^{\{u,v\}},\ldots, x_{n-2}^{\{u,v\}}$ for each $P^{\{u,v\}}$. {{We}} use vertices $u$, $v$ and $w$ that
are {{connected as in}} Figure \ref{edge1nodd}.
We only can remove  $x_1^{\{v,w\}}$ and $x_{n-1}^{\{u,v\}}$ from $D$ and put $v$ in it to have a dominating set. So $D'=D-\{x_1^{\{v,w\}},x_{n-1}^{\{u,v\}}\}\cup\{v\}$ is a dominating set which clearly is not a secure dominating set. 
{{By the same argument for $n=7k+r$, where $k$ is a positive even number and $r \in\{-1,1,3,5\}$.}}
\qed
	\end{proof}

	\begin{figure}
		\begin{center}
			\psscalebox{0.7 0.7}
{
\begin{pspicture}(0,-5.1593056)(17.997116,-4.0379167)
\psdots[linecolor=black, dotsize=0.4](0.19711533,-4.2393055)
\psdots[linecolor=black, dotsize=0.4](12.197115,-4.2393055)
\psline[linecolor=black, linewidth=0.08](0.19711533,-4.2393055)(6.997115,-4.2393055)(6.997115,-4.2393055)
\psline[linecolor=black, linewidth=0.08](8.5971155,-4.2393055)(12.197115,-4.2393055)(12.197115,-4.2393055)
\psdots[linecolor=black, dotstyle=o, dotsize=0.4, fillcolor=white](10.5971155,-4.2393055)
\psdots[linecolor=black, dotstyle=o, dotsize=0.4, fillcolor=white](8.997115,-4.2393055)
\psdots[linecolor=black, dotstyle=o, dotsize=0.4, fillcolor=white](6.5971155,-4.2393055)
\psdots[linecolor=black, dotstyle=o, dotsize=0.4, fillcolor=white](4.997115,-4.2393055)
\psdots[linecolor=black, dotstyle=o, dotsize=0.4, fillcolor=white](3.3971152,-4.2393055)
\psdots[linecolor=black, dotstyle=o, dotsize=0.4, fillcolor=white](1.7971153,-4.2393055)
\psdots[linecolor=black, dotsize=0.1](7.397115,-4.2393055)
\psdots[linecolor=black, dotsize=0.1](7.7971153,-4.2393055)
\psdots[linecolor=black, dotsize=0.1](8.197115,-4.2393055)
\rput[bl](0.037115324,-5.0193057){$u$}
\rput[bl](12.057116,-5.0793056){$v$}
\rput[bl](1.2571154,-5.0793056){$x_1^{\{u,v\}}$}
\rput[bl](2.9971154,-5.0793056){$x_2^{\{u,v\}}$}
\rput[bl](4.437115,-5.1593056){$x_3^{\{u,v\}}$}
\rput[bl](6.0971155,-5.1193056){$x_4^{\{u,v\}}$}
\rput[bl](8.497115,-5.1193056){$x_{n-2}^{\{u,v\}}$}
\rput[bl](10.1371155,-5.1593056){$x_{n-1}^{\{u,v\}}$}
\psline[linecolor=black, linewidth=0.08](12.197115,-4.2393055)(14.197115,-4.2393055)(13.797115,-4.2393055)
\psline[linecolor=black, linewidth=0.08](15.797115,-4.2393055)(17.397116,-4.2393055)(17.397116,-4.2393055)
\psdots[linecolor=black, dotsize=0.1](14.5971155,-4.2393055)
\psdots[linecolor=black, dotsize=0.1](14.997115,-4.2393055)
\psdots[linecolor=black, dotsize=0.1](15.397116,-4.2393055)
\psdots[linecolor=black, dotsize=0.4](17.797115,-4.2393055)
\psdots[linecolor=black, dotstyle=o, dotsize=0.4, fillcolor=white](13.797115,-4.2393055)
\psdots[linecolor=black, dotstyle=o, dotsize=0.4, fillcolor=white](16.197115,-4.2393055)
\psline[linecolor=black, linewidth=0.08](17.397116,-4.2393055)(17.797115,-4.2393055)(17.797115,-4.2393055)
\rput[bl](17.697115,-4.9993057){$w$}
\rput[bl](13.357116,-5.0993056){$x_1^{\{v,w\}}$}
\rput[bl](15.777115,-5.1193056){$x_{n-1}^{\{v,w\}}$}
\end{pspicture}
}
		\end{center}
		\caption{Superedge $P^{\{u,v\}}$ and $P^{\{v,w\}}$ in $G^{\frac{1}{n}}$ related to the proof of Theorem \ref{G16}} \label{edge1nodd}
	\end{figure}
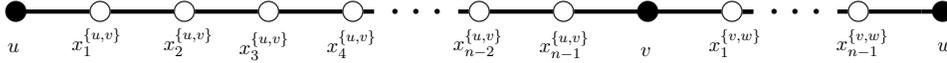

Now, we consider the other cases and present some bounds for them:

\begin{theorem}
Let $G$ be a graph and $n=7k+r$, where $k$ is a positive number and $r \in\{0,2,4\}$. Then,
 $$|V(G)|+ \gamma_s(P_{n-3})|E(G)| \leq\gamma_s(G^{\frac{1}{n}})\leq\gamma_s(P_{n+1})|E(G)|  .$$
	\end{theorem}

\begin{proof}
First {{consider}} every superedge $P^{\{u,v\}}$ and choose a secure dominating set for {{it}}, and then put all these vertices in a set. Therefore we have a secure dominating set for $G$ with size at most $\gamma_s(P_{n+1})|E(G)|  $. Now, for finding a lower bound, first we put every vertex of $G$ in our set $D$ but we do not consider their neighbours. Then we have a path $P_{n-3}$ for every superedge and we should choose between these vertices. Obviously we can choose  $\gamma_s(P_{n-3})$ of these vertices and add to $D$. {{Clearly, $D$ is a dominating set which is not necessarily secure in some cases}}. Now, by Proposition \ref{COC-pro}, we conclude that $\gamma_s(G^{\frac{1}{n}})\geq |V(G)|+ \gamma_s(P_{n-3})|E(G)|$, and therefore we have the result.
\qed
	\end{proof}

\medskip

At the beginning of this section, we presented a sharp upper bound for secure domination number of $G^{\frac{1}{2}}$ in Theorem  \ref{G12}. There are some graphs $G$, which show that $\gamma_s(G^{\frac{1}{2}})< min \{ |E(G)|, |V(G)|\}$. For example, consider path graph $P_{11}$. Then $P_{11}^{\frac{1}{2}}=P_{21}$, and by Theorem \ref{COC}, $\gamma_s(P_{11}^{\frac{1}{2}})=9$. So, this inspires us to find a lower bound for $\gamma_s(G^{\frac{1}{2}})$. We end this section with the following conjecture:

\begin{conjecture}\label{Conj}
For every graph $G$, $\gamma_s(G^{\frac{1}{2}})>\frac{4}{5}|V(G)|$.
\end{conjecture}

\section{Conclusions}

In this paper, we obtained the secure domination number of $k$-subdivision of graphs for some cases and presents some lower and upper bounds for other ones. 
Future topics of interest for future research include the following suggestions:

\begin{itemize}
\item[•]
Proving Conjecture \ref{Conj} or finding a better lower bound for $\gamma_s(G^{\frac{1}{2}})$.
\item[•]
What is the exact value of $\gamma_s(G^{\frac{1}{n}})$ for  $n=7k+r$, where $k$ is a positive integer value and $r \in\{0,2,4\}$?
\end{itemize}


\end{document}